\pdfoutput=1
\documentclass[a4paper,11pt]{amsart}

\usepackage[margin=1in,includehead,includefoot]{geometry}
\usepackage[utf8]{inputenc}
\usepackage[T1]{fontenc}
\usepackage{lmodern,microtype}
\usepackage{mathtools,amsthm,amssymb}
\usepackage{enumitem}
\usepackage{tikz}
\usepackage{hyperref}
\usepackage{bookmark}

\hypersetup{colorlinks=true,linkcolor=black,citecolor=black,filecolor=black,urlcolor=black}

\makeatletter
\let\citationorig\citation
\def\citation#1{\citationorig{#1}\@for\@tempa:=#1\do{\@ifundefined{cit@\@tempa}{\global\@namedef{cit@\@tempa}{}}{}}}
\let\bibitemorig\bibitem
\def\bibitem#1{\@ifundefined{cit@#1}{\typeout{LaTeX Warning: Unused bibitem `#1'}}{}\bibitemorig{#1}}
\let\old@setaddresses\@setaddresses
\def\@setaddresses{\bigskip{\parindent 0pt\let\scshape\relax\let\ttfamily\relax\old@setaddresses}}
\makeatother

\def\periodsf{\spacefactor 3000 \space}

\newtheorem{theorem}{Theorem}
\newtheorem{lemma}[theorem]{Lemma}
\theoremstyle{remark}
\newtheorem{condition}{Condition}
\newtheorem{claim}[condition]{Claim}
\newtheorem{reduction}{Reduction to Condition}
\newtheorem{claimproof}[reduction]{Proof of Claim}

\renewenvironment{itemize}{\begin{itemorig}[label=\textbullet, noitemsep, topsep=3pt plus 3pt, labelsep=.6em, labelindent=.2em, leftmargin=*]}{\end{itemorig}}

\def\famF{\mathcal{F}}
\def\famH{\mathcal{H}}
\def\famL{\mathcal{L}}
\def\famS{\mathcal{S}}

\let\leq\leqslant

\let\setminus\smallsetminus

\allowdisplaybreaks

\hypersetup{
  pdftitle={Coloring triangle-free L-graphs with O(log log n) colors},
  pdfauthor={Bartosz Walczak},
}

\title[Coloring triangle-free L-graphs with $O(\log\log n)$ colors]{\boldmath Coloring triangle-free L-graphs with $O(\log\log n)$ colors}
\author{Bartosz Walczak}
\address{Department of Theoretical Computer Science, Faculty of Mathematics and Computer Science, Jagiellonian University, Kraków, Poland}
\email{\href{mailto:walczak@tcs.uj.edu.pl}{walczak@tcs.uj.edu.pl}}

\thanks{An extended abstract of this paper appeared in \href{http://www.iam.fmph.uniba.sk/amuc/ojs/index.php/amuc/article/view/1255}{\emph{Acta Math.\ Univ.\ Comenianae} 88~(3), 1063--1069, 2019}.}
\thanks{Work partially supported by National Science Center of Poland grant 2015/17/D/ST1/00585.}

\begin{document}

\begin{abstract}
It is proved that triangle-free intersection graphs of $n$ L-shapes in the plane have chromatic number $O(\log\log n)$.
This improves the previous bound of $O(\log n)$ (McGuinness, 1996) and matches the known lower bound construction (Pawlik et~al., 2013).
\end{abstract}

\maketitle

\section{Introduction}

The \emph{intersection graph} of a family of sets $\famF$ has these sets as vertices and the pairs of the sets that intersect as edges.
An \emph{L-shape} is a set in the plane formed by one horizontal segment and one vertical segment joined at the left endpoint of the former and the bottom endpoint of the latter, as in the letter L\@.
An \emph{L-graph} is an intersection graph of L-shapes.
A stretching argument of Middendorf and Pfeiffer \cite{MP92} shows that L-graphs form a subclass of the \emph{segment graphs}, that is, intersection graphs of straight-line segments in the plane.
Segment graphs form a subclass of the \emph{string graphs}---intersection graphs of generic curves in the plane.

L-graphs are perhaps not as natural as segment graphs, but they capture a lot of complexity of segment graphs while being significantly easier to deal with.
For instance, the famous result of Chalopin and Gonçalves \cite{CG09} that all planar graphs are segment graphs (solution to Scheinerman's conjecture \cite{Sch84}) was recently strengthened by Gonçalves, Isenmann, and Pennarun \cite{GIP18} who showed, with a much simpler and more elegant argument, that all planar graphs are L-graphs.
Other recent works on L-graphs and their relation to other classes of graphs include \cite{ADD+-arxiv,FKMU16,JT19}.

Our concern in this paper is how large the chromatic number $\chi$ can be in terms of the number of vertices $n$ for triangle-free geometric intersection graphs.
Classical constructions of triangle-free graphs with arbitrarily large chromatic number such as Mycielski graphs \cite{Myc55} and shift graphs \cite{EH64} achieve $\chi=\varTheta(\log n)$ but are not realizable as string graphs.
Non-constructive methods even provide triangle-free graphs with $\chi=\varTheta(\sqrt{n/\log n})$ \cite{Kim95}.
However, for triangle-free geometric intersection graphs (string graphs), even to determine whether the chromatic number can grow arbitrarily high was a long-standing open problem.
Raised in the 1980s by Erdős for segment graphs (see \cite{Gya87}) and by Gyárfás and Lehel \cite{GL85} for L-graphs, it was solved only recently \cite{PKK+13,PKK+14}.

\begin{theorem}[Pawlik et~al.\ \cite{PKK+13}]
\label{thm:pawlik}
There exist triangle-free L-graphs with\/ $\chi=\varTheta(\log\log n)$.
\end{theorem}

By contrast, various classes of geometric intersection graphs are \emph{$\chi$-bounded}, which means that graphs of these classes have chromatic number bounded by some function of the clique number.
So are, for instance, the classes of rectangle graphs (intersection graphs of axis-parallel rectangles) \cite{AG60} and circle graphs (intersection graphs of chords of a circle) \cite{Gya85}.
The latter is a subclass of the class of infinite-L-graphs (intersection graphs of L-shapes whose vertical parts are upward-infinite), proved to be $\chi$-bounded by McGuinness \cite{McG96}.
A very general result extending this and several later ones \cite{LMPW14,McG00,RW19b,Suk14} was proved by Rok and Walczak \cite{RW19a}: the class of intersection graphs of curves each crossing a fixed curve at least once and at most $t$ times is $\chi$-bounded.
See Scott and Seymour's excellent survey \cite{SS-arxiv} for many more recent results on $\chi$-boundedness.

Key insights that led to the construction in Theorem~\ref{thm:pawlik} came from analyzing \emph{directed frame graphs}, that is, intersection graphs of \emph{frames} (boundaries of axis-parallel rectangles) whose top sides are free of intersections with other frames.
They form a subclass of the L-graphs.
The construction in Theorem~\ref{thm:pawlik} actually provides triangle-free directed frame graphs with $\chi=\varTheta(\log\log n)$.
A more general construction in \cite{KW17} produces string graphs with chromatic number $\varTheta_\omega((\log\log n)^{\omega-1})$, which however are not representable as frame graphs, L-graphs, or segment graphs.
It is natural to ask whether the bounds on $\chi$ with respect to $n$ achieved by these constructions are (close to) optimal.
This is true for triangle-free directed frame graphs.

\begin{theorem}[Krawczyk et~al.\ \cite{KPW15}]
\label{thm:frames}
Triangle-free directed frame graphs satisfy\/ $\chi=O(\log\log n)$.
\end{theorem}

\noindent
The same bound also holds for general triangle-free frame graphs (not necessarily directed) \cite{KPW15}.

The contribution of this paper is the following generalization of Theorem~\ref{thm:frames}.

\begin{theorem}
\label{thm:main}
Triangle-free L-graphs satisfy\/ $\chi=O(\log\log n)$.
\end{theorem}

\noindent
It shows that the construction witnessing Theorem~\ref{thm:pawlik} is asymptotically optimal also for triangle-free L-graphs.
The previous best bound was $\chi=O(\log n)$; it follows from the above-mentioned result of McGuinness \cite{McG96} on infinite-L-graphs by ``divide-and-conquer'', and it holds also for L-graphs with clique number bounded by any constant.
When the clique number is $\omega$, the best known upper bounds on the chromatic number are $O((\log\log n)^{\omega-1})$ for frame graphs \cite{KW17}, $O(\log n)$ for segment graphs \cite{Suk14}, and $(\log n)^{\smash{O(\log\omega)}}$ for string graphs in general \cite{FP14}.
It remains open whether a double-logarithmic upper bound on the chromatic number holds, for instance, for triangle-free segment graphs or for L-graphs with bounded clique number.

\section{Preliminaries}

Graph-theoretic terms like \emph{chromatic number} and \emph{triangle-free} applied directly to a family of curves $\famF$ have the same meaning as when applied to the intersection graph of $\famF$.
A family of curves $\famF$ is \emph{$1$-intersecting} if any two curves in $\famF$ have at most one common point.
For $c\in\famF$, let $\famF(c)$ denote the family of curves at distance exactly $2$ from $c$ in the intersection graph of $\famF$.

\begin{lemma}[McGuinness {\cite[Theorem~5.3]{McG01}}]
\label{lem:radius2}
There is a constant\/ $\alpha>0$ such that every triangle-free\/ $1$-intersecting family of curves\/ $\famF$ satisfies\/ $\chi(\famF)\leq\alpha\max_{c\in\famF}\chi(\famF(c))$.
\end{lemma}

\noindent
The statement of Lemma~\ref{lem:radius2} can be generalized to families of curves with unrestricted intersections and with clique number bounded by any constant (which $\alpha$ depends on); see \cite{CSS-arxiv}.

A point $p$ lies \emph{above}/\emph{below}/\emph{to the left}/\emph{right of} a plane set $s$ (an L-shape, a segment, a line, etc.)\ if $p\notin s$ and the ray emanating from $p$ downwards/upwards/rightwards/leftwards intersects~$s$.
A~plane set $r$ lies \emph{above}/\emph{below}/\emph{to the left}/\emph{right of} a plane set $s$ if so does every point of $r$.

Let $h$ be a horizontal line.
A curve $c$ is \emph{grounded to} $h$ when one endpoint of $c$ lies on $h$ and the remaining part of $c$ lies above $h$.

\begin{lemma}[McGuinness \cite{McG00}]
\label{lem:grounded}
Triangle-free\/ $1$-intersecting families of curves grounded to a fixed horizontal line\/ $h$ have bounded chromatic number.
\end{lemma}

\noindent
Again, the statement of Lemma~\ref{lem:grounded} can be generalized to families of grounded curves with unrestricted intersections and with clique number bounded by any constant \cite{RW19b}.

An \emph{$h$-even-curve} is a curve that starts above $h$ and crosses $h$ properly a positive even number of times, ending again above $h$.
The two parts of an $h$-even-curve $c$ from an endpoint to the first intersection point with $h$ are denoted by $L(c)$ and $R(c)$ so that the common point of $L(c)$ with $h$ is to the left of that of $R(c)$.
A family of $h$-even-curves is an \emph{$LR$-family} if every intersection between two of its members $c_1$ and $c_2$ is between $L(c_1)$ and $R(c_2)$ or vice versa.

\begin{lemma}[Rok, Walczak {\cite[Theorem~4]{RW19a}}]
\label{lem:even}
For a fixed horizontal line\/ $h$, triangle-free\/ $LR$-families of\/ $h$-even-curves have bounded chromatic number.
\end{lemma}

\noindent
Again, the statement of Lemma~\ref{lem:even} can be generalized to $LR$-families of $h$-even-curves with clique number bounded by any constant \cite[Theorem~4]{RW19a}.

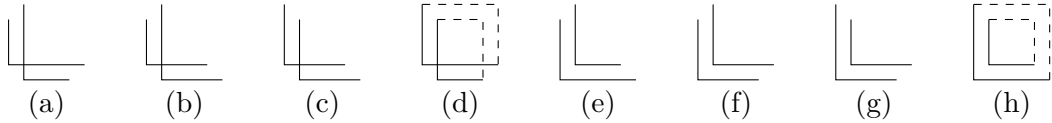
\begin{figure}[t]
\def\interskip{\hskip 0.8cm\relax}
\centering
\begin{tikzpicture}
  \draw (0,0.8)|-(1,0.2);
  \draw (0.2,1)|-(0.8,0);
  \node[below] at (0.5,0) {(a)};
\end{tikzpicture}\interskip
\begin{tikzpicture}
  \draw (0,0.8)|-(0.8,0.2);
  \draw (0.2,1)|-(1,0);
  \node[below] at (0.5,0) {(b)};
\end{tikzpicture}\interskip
\begin{tikzpicture}
  \draw (0,1)|-(0.8,0.2);
  \draw (0.2,0.8)|-(1,0);
  \node[below] at (0.5,0) {(c)};
\end{tikzpicture}\interskip
\begin{tikzpicture}
  \draw (0,1)|-(1,0.2);
  \draw (0.2,0.8)|-(0.8,0);
  \draw[dashed] (0,1)-|(1,0.2);
  \draw[dashed] (0.2,0.8)-|(0.8,0);
  \node[below] at (0.5,0) {(d)};
\end{tikzpicture}\interskip
\begin{tikzpicture}
  \draw (0,0.8)|-(1,0);
  \draw (0.2,1)|-(0.8,0.2);
  \node[below] at (0.5,0) {(e)};
\end{tikzpicture}\interskip
\begin{tikzpicture}
  \draw (0,0.8)|-(0.8,0);
  \draw (0.2,1)|-(1,0.2);
  \node[below] at (0.5,0) {(f)};
\end{tikzpicture}\interskip
\begin{tikzpicture}
  \draw (0,1)|-(0.8,0);
  \draw (0.2,0.8)|-(1,0.2);
  \node[below] at (0.5,0) {(g)};
\end{tikzpicture}\interskip
\begin{tikzpicture}
  \draw (0,1)|-(1,0);
  \draw (0.2,0.8)|-(0.8,0.2);
  \draw[dashed] (0,1)-|(1,0);
  \draw[dashed] (0.2,0.8)-|(0.8,0.2);
  \node[below] at (0.5,0) {(h)};
\end{tikzpicture}
\vspace*{-0.5ex}
\caption{Configurations of pairs L-shapes with intersecting bounding boxes;
the case of right endpoints with common $x$-coordinate is included in (b), (c), (e), (h), and the case of top endpoints with common $y$-coordinate in (c), (d), (g), (h)}
\label{fig:conf}
\end{figure}

Let $h(\ell)$ and $v(\ell)$ denote the horizontal and the vertical segment of an L-shape $\ell$, respectively.
Their intersection point is the \emph{corner} of $\ell$, and their other endpoints are the \emph{right} and the \emph{top endpoint} of $\ell$, respectively.
We will be assuming that the horizontal segments of all L-shapes that we consider have distinct $y$-coordinates and the vertical ones have distinct $x$-coordinates.

The \emph{bounding box} of an L-shape $\ell$ is the minimal axis-parallel rectangle containing $\ell$.
Figure~\ref{fig:conf} illustrates all possible configurations of pairs of L-shapes whose minimal bounding boxes intersect.
If a family of L-shapes $\famL$ contains no pairs in configurations (a)--(c) and (e)--(g), then completing the L-shapes in $\famL$ to frames (by taking the boundaries of their bounding boxes) and appropriate perturbation of collinear sides shows that the intersection graph of $\famL$ is a directed frame graph.
This leads to the following corollary to Theorem~\ref{thm:frames}.

\begin{lemma}
\label{lem:frames}
Triangle-free families of L-shapes with no pairs in configurations\/ \textup{(a)--(c)} and\/ \textup{(e)--(g)} satisfy\/ $\chi(\famL)=O(\log\log n)$.
\end{lemma}

\section{Key Lemma}

\begin{lemma}
\label{lem:key}
Let\/ $h$ be a horizontal line, $\famL$ be a family of L-shapes lying above\/ $h$, and\/ $\famS$ be a family of vertical segments starting at\/ $h$ and going upwards such that
\begin{itemize}
\item the horizontal segments of the L-shapes in\/ $\famL$ have distinct\/ $y$-coordinates,
\item the vertical segments in\/ $\{v(\ell)\colon\ell\in\famL\}\cup\famS$ have distinct\/ $x$-coordinates,
\item every L-shape in\/ $\famL$ intersects some vertical segment in\/ $\famS$,
\item the intersection graph of\/ $\famL\cup\famS$ is triangle-free.
\end{itemize}
Then\/ $\chi(\famL)=O(\log\log{|\famL|})$.
\end{lemma}

\begin{proof}
We will reduce the problem to the case where configurations (a)--(c) and (e)--(g) from Figure~\ref{fig:conf} are excluded, in order to apply Lemma~\ref{lem:frames}.
The reduction will keep modifying $\famL$ and $\famS$ to make them satisfy more and more additional conditions.
We will present each step of the reduction by first formulating a new condition that $\famL$ and $\famS$ should satisfy and then explaining how to modify $\famL$ and $\famS$ to ensure that condition while preserving all previous conditions and \emph{changing\/ $\chi(\famL)$ by at most a constant factor}.
The latter guarantees that the bound $\chi(\famL)=O(\log\log{|\famL|})$ after the reduction implies the same bound for the original family $\famL$.

The \emph{leftmost}/\emph{rightmost support} of an L-shape $\ell\in\famL$ is the vertical segment in $\famS$ with minimum/maximum $x$-coordinate among all vertical segments in $\famS$ that intersect $\ell$.
The \emph{handle} of an L-shape $\ell\in\famL$ is the part of $\ell$ to the left of the vertical line containing the leftmost support of $\ell$.
The \emph{hook} of an L-shape $\ell\in\famL$ is the part of $\ell$ to the right of the rightmost support of $\ell$.

\begin{condition}
\label{cond:handle}
No two handles of L-shapes in $\famL$ intersect.
\end{condition}

\begin{reduction}
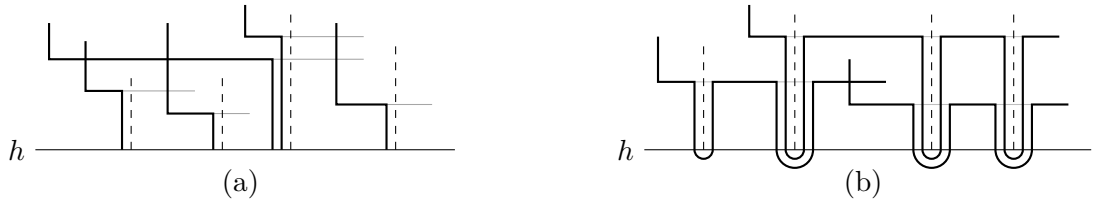
\begin{figure}[t]
\centering
\begin{tikzpicture}[scale=0.6]
  \draw (0.4,0)--(9.6,0);
  \draw[black!35] (1.5,1.3)--(3.9,1.3);
  \draw[black!35] (4.3,0.8)--(5.1,0.8);
  \draw[black!35] (5.6,2)--(7.6,2);
  \draw[black!35] (5.8,2.5)--(7.6,2.5);
  \draw[black!35] (8.1,1)--(9.1,1);
  \draw[thick] (1.5,2.4)--(1.5,1.3)--(2.3,1.3)--(2.3,0);
  \draw[thick] (3.3,2.8)--(3.3,0.8)--(4.3,0.8)--(4.3,0);
  \draw[thick] (0.7,2.8)--(0.7,2)--(5.6,2)--(5.6,0);
  \draw[thick] (5,3.2)--(5,2.5)--(5.8,2.5)--(5.8,0);
  \draw[thick] (7,2.8)--(7,1)--(8.1,1)--(8.1,0);
  \draw[dashed] (2.5,0)--(2.5,1.6);
  \draw[dashed] (4.5,0)--(4.5,1.6);
  \draw[dashed] (6,0)--(6,3);
  \draw[dashed] (8.3,0)--(8.3,2.3);
  \node[left] at (0.4,0) {$h$};
  \node[below] at (4.9,-0.2) {(a)};
\end{tikzpicture}\hskip 2cm
\begin{tikzpicture}[scale=0.6]
  \draw (0.2,0)--(10,0);
  \draw[black!25] (1.3,1.5)--(1.7,1.5) (3.1,1.5)--(3.9,1.5);
  \draw[black!25] (6.1,1)--(6.9,1) (7.9,1)--(8.7,1);
  \draw[black!25] (3.3,2.5)--(3.7,2.5) (6.3,2.5)--(6.7,2.5) (8.1,2.5)--(8.5,2.5);
  \draw[thick] (0.5,2.5)--(0.5,1.5)--(1.3,1.5)--(1.3,0) arc (-180:0:2mm)--(1.7,1.5)--(3.1,1.5)--(3.1,0) arc (-180:0:4mm)--(3.9,1.5)--(5.5,1.5);
  \draw[thick] (4.7,2)--(4.7,1)--(6.1,1)--(6.1,0) arc (-180:0:4mm)--(6.9,1)--(7.9,1)--(7.9,0) arc (-180:0:4mm)--(8.7,1)--(9.5,1);
  \draw[thick] (2.5,3.2)--(2.5,2.5)--(3.3,2.5)--(3.3,0) arc (-180:0:2mm)--(3.7,2.5)--(6.3,2.5)--(6.3,0) arc (-180:0:2mm)--(6.7,2.5)--(8.1,2.5)--(8.1,0) arc (-180:0:2mm)--(8.5,2.5)--(9.3,2.5);
  \draw[dashed] (1.5,0)--(1.5,2.3);
  \draw[dashed] (3.5,0)--(3.5,3);
  \draw[dashed] (6.5,0)--(6.5,3);
  \draw[dashed] (8.3,0)--(8.3,3);
  \node[left] at (0.2,0) {$h$};
  \node[below] at (5,-0.2) {(b)};
\end{tikzpicture}
\vspace*{-2ex}
\caption{Transformations of L-shapes to curves considered in the proof}
\label{fig:auxiliary}
\end{figure}

The family $\famH$ of handles of the L-shapes in $\famL$ can be transformed into a $1$-intersecting family of curves grounded to $h$ with the same intersection graph by connecting them to $h$ along the leftmost supports, as illustrated in Figure \ref{fig:auxiliary}~(a).
Therefore, by Lemma~\ref{lem:grounded}, $\chi(\famH)$ is bounded.
Let $\famL_c$ denote the L-shapes in $\famL$ whose handles have color $c$ in an optimal proper coloring of $\famH$.
Then $\chi(\famL)\leq\sum_c\chi(\famL_c)\leq\chi(\famH)\max_c\chi(\famL_c)=O(\max_c\chi(\famL_c))$.
We set $\famL:=\famL_c$ for the color $c$ that maximizes $\chi(\famL_c)$, and Condition~\ref{cond:handle} follows.
\end{reduction}

\begin{condition}
\label{cond:hook}
All of the L-shapes in $\famL$ have empty hooks.
Consequently, no two L-shapes in $\famL$ occur in configuration (b) or (c) from Figure~\ref{fig:conf}.
\end{condition}

\begin{reduction}
For $\ell\in\famL$, let $\ell'$ be the L-shape obtained from $\ell$ by cutting the hook off, and let $\famL'=\{\ell'\colon\ell\in\famL\}$.
Let $\famL_c$ be the L-shapes $\ell\in\famL$ such that $\ell'$ has color $c$ in an optimal proper coloring of $\famL'$.
Since every crossing between two L-shapes in $\famL_c$ is a crossing between the vertical part of one and the hook of the other, the family $\famL_c$ can be transformed into an $LR$-family of $h$-even-curves with the same intersection graph, as illustrated in Figure \ref{fig:auxiliary}~(b).
Therefore, by Lemma~\ref{lem:even}, $\chi(\famL_c)$ is bounded.
This yields $\chi(\famL)\leq\sum_c\chi(\famL_c)\leq\chi(\famL')\max_c\chi(\famL_c)=O(\chi(\famL'))$.
We set $\famL:=\famL'$, and Condition~\ref{cond:hook} follows: if two L-shapes $\ell_1,\ell_2\in\famL'$ occurred in configuration (b) or (c) from Figure~\ref{fig:conf}, they would form a triangle with the rightmost support of either $\ell_1$ or $\ell_2$.
\end{reduction}

\begin{condition}
\label{cond:below}
No corner or endpoint of an L-shape in $\famL$ lies below the handle of an L-shape in $\famL$.
\end{condition}

\begin{reduction}
\begin{figure}[t]
\begin{tikzpicture}[scale=.7]
  \draw[dotted] (-2,1.5)--(-2,1);
  \draw[dotted] (0,-2.45)--(0,-2);
  \draw[thick] (-2,1)|-(1,0);
  \node[fill=white,inner sep=1pt] at (-1.3,0) {$\ell$};
  \draw (-1.8,-0.3)|-(1,-0.9);
  \draw (-3.3,0.1)|-(-1.5,-0.5);
  \draw (-2.5,1)|-(-1,-1.3);
  \draw (-0.7,-0.5)|-(-0.2,-1.5);
  \draw (-3.2,1.2)|-(-0.5,0.7);
  \draw (-1,1.2)|-(1,0.4);
  \draw (0.7,0.9)|-(1.3,-1.5);
  \draw[dashed] (-2.9,-2)--(-2.9,1.2);
  \draw[dashed] (-1.3,-2)--(-1.3,-0.6);
  \draw[dashed] (-0.45,-2)--(-0.45,-1.2);
  \draw[dashed] (0,-2)--(0,1);
  \draw[dashed] (1,-2)--(1,-1.2);
  \draw (-3.5,-2) node[left]{$h$}--(1.5,-2);
  \draw[gray,ultra thick,->] (2.8,-0.495)--(4,-0.495);
  \fill[black!10] (7,1.5)--(7,0)--(9,0)--(9,-2.45)--(11.5,-2.45)--(11.5,0)--(9.5,0)--(9.5,1.5)--cycle;
  \draw[dotted] (7,1.5)--(7,0)--(9,0)--(9,-2);
  \draw[dotted] (9.5,1.5)--(9.5,1);
  \draw[dotted] (9,-2.45)--(9,-2);
  \draw[dotted] (11.5,-2.45)--(11.5,-2);
  \draw[thick] (9.5,1)|-(12.5,0);
  \node[fill=white,inner sep=1pt] at (10.2,0) {$\ell$};
  \draw (7.2,-0.3)|-(12.5,-0.9);
  \draw (5.7,0.1)|-(7.5,-0.5);
  \draw (6.5,1)|-(8,-1.3);
  \draw (8.3,-0.5)|-(8.8,-1.5);
  \draw (5.8,1.2)|-(11,0.7);
  \draw (10.5,1.2)|-(12.5,0.4);
  \draw (12.2,0.9)|-(12.8,-1.5);
  \draw[dashed] (6.1,-2)--(6.1,1.2);
  \draw[dashed] (7.7,-2)--(7.7,-0.6);
  \draw[dashed] (8.55,-2)--(8.55,-1.2);
  \draw[dashed] (11.5,-2)--(11.5,1);
  \draw[dashed] (12.5,-2)--(12.5,-1.2);
  \draw (5.5,-2) node[left]{$h$}--(13,-2);
\end{tikzpicture}
\vspace*{-1.5ex}
\caption{Horizontal shift considered in the proof}
\label{fig:shift}
\end{figure}
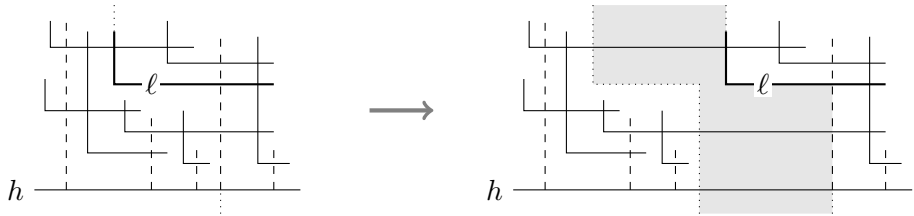

Let $\ell\in\famL$.
Let $s$ be part of $h(\ell)$ from the corner of $\ell$ to the intersection point with the leftmost support of $\ell$.
By Condition~\ref{cond:handle}, no other L-shape in $\famL$ intersects $s$.
Cut the plane along $s$ and the rays emanating from the left endpoint of $s$ upwards and from the right endpoint of $s$ downwards, and shift the two parts horizontally, as illustrated in Figure~\ref{fig:shift}, keeping $\ell$ and the leftmost support of $\ell$ in the right part and adding appropriate horizontal connections to the L-shapes in $\famL$ that have been cut through.
This ensures that no corner or endpoint of an L-shape in $\famL$ lies below the handle of $\ell$, while preserving Conditions \ref{cond:handle} and~\ref{cond:hook}.
Condition~\ref{cond:below} is ensured by performing such a shift for every $\ell\in\famL$.
\end{reduction}

\begin{condition}
\label{cond:left1}
No corner of an L-shape in $\famL$ lies to the left of two intersecting members of $\famL\cup\famS$.
\end{condition}

\begin{reduction}
Let $\ell\in\famL$ and $\ell'\in\famL\cup\famS$ be such that $\ell'$ crosses $h(\ell)$.
Let $s$ be the part of $v(\ell)$ that lies to the left of $\ell'$.
No L-shape in $\famL$ can cross $s$, otherwise (by Condition \ref{cond:hook}) it would also cross $\ell'$, thus creating a triangle.
Cut the plane along $s$ and the rays emanating from the top endpoint of $s$ leftwards and from the bottom endpoint of $s$ rightwards, and shift the two parts vertically keeping $\ell$ in the upper part and adding appropriate vertical connections to the L-shapes in $\famL$ and the vertical segments in $\famS$ that have been cut through, in a way analogous to a horizontal shift illustrated in Figure~\ref{fig:shift}.
This ensures that no corner of an L-shape in $\famL$ lies to the left of $s$, while preserving Conditions \ref{cond:handle}--\ref{cond:below}.
Condition~\ref{cond:left1} is ensured by performing such a shift for any $\ell\in\famL$ and $\ell'\in\famL\cup\famS$ such that $\ell'$ crosses $h(\ell)$.
\end{reduction}

\begin{condition}
\label{cond:left2}
If $\ell\in\famL$ and $\ell'\in\famL\cup\famS$ intersect, then $v(\ell)$ does not lie entirely to the left of $\ell'$.
In~particular, no two L-shapes in $\famL$ occur in configuration (a) from Figure~\ref{fig:conf}.
\end{condition}

\begin{reduction}
Let $\famL^*$ be the L-shapes $\ell\in\famL$ for which there is $\ell'\in\famL\cup\famS$ intersecting $\ell$ such that $v(\ell)$ lies to the left of $\ell'$.
For any $\ell\in\famL^*$, by Condition~\ref{cond:left1}, no corner of an L-shape in $\famL$ lies to the left of $\ell$, which implies that no L-shape in $\famL$ intersect $v(\ell)$.
Therefore, the L-shapes in $\famL^*$ are pairwise disjoint, so $\chi(\famL)\leq\chi(\famL\setminus\famL^*)+1$.
We set $\famL:=\famL\setminus\famL^*$, and Condition~\ref{cond:left2} follows.
\end{reduction}

\begin{claim}
\label{claim}
Let $\ell_1,\ell_2,\ell\in\famL$.
If $\ell_1$ and $\ell_2$ are in configuration (e) or (f) from Figure~\ref{fig:conf} and the corner of $\ell_1$ lies below $h(\ell)$, then the corner of $\ell_2$ also lies below $h(\ell)$.
\end{claim}

\begin{claimproof}
Suppose that not all of $h(\ell_1)$ lies below $h(\ell)$.
The L-shapes $\ell$ and $\ell_1$ cannot intersect, as they would contradict Condition~\ref{cond:hook}, so $v(\ell_1)$ lies below $h(\ell)$.
This and Condition~\ref{cond:hook} imply that $v(\ell_1)$ lies to the left of the rightmost support of $\ell$, which contradicts Condition~\ref{cond:left2}.
This shows that $h(\ell_1)$ lies below $h(\ell)$, so the corner of $\ell_2$ lies above or below $h(\ell)$.
If it lies above $h(\ell)$, then $\ell$ and $\ell_1$ intersect, so (by Condition~\ref{cond:below}) the leftmost support of $\ell_2$ intersects both $\ell$ and $\ell_1$, forming a triangle.
This contradiction shows that the corner of $\ell_2$ lies below $h(\ell)$.
\end{claimproof}

\begin{condition}
\label{cond:vertical}
No two L-shapes in $\famL$ occur in configuration (e) or (f) from Figure \ref{fig:conf}.
\end{condition}

\begin{reduction}
Suppose $\ell_1,\ell_2\in\famL$ occur in configuration (e) or (f), where $\ell_1$ has lower horizontal segment.
Pull the top endpoint of $\ell_1$ up onto the horizontal line containing the top endpoint of $\ell_2$ to obtain an L-shape $\ell_1'$ such that $h(\ell_1)=h(\ell_1')$ and $v(\ell_1)\subset v(\ell_1')$.
Let $\famL'=\famL\setminus\{\ell_1\}\cup\{\ell_1'\}$.
Clearly, $\famL'$ satisfies Conditions \ref{cond:handle}--\ref{cond:below}.
If $h(\ell_1)$ crosses some $\ell'\in\famL\cup\famS$, then (by Condition~\ref{cond:left2}) the top endpoint of $\ell_1$ is higher than the top endpoint of $\ell'$.
Therefore, for any $\ell'\in\famL\cup\famS$, Condition~\ref{cond:left1} for the pair $(\ell_1',\ell')$ follows from Condition~\ref{cond:left1} for the pair $(\ell_1,\ell')$.
If $v(\ell_1')$ crosses $h(\ell)$ for some $\ell\in\famL$, then (by Claim~\ref{claim}) the corner of $\ell_2$ lies below $h(\ell)$, so (since the top endpoints of $\ell_1'$ and $\ell_2$ lie on the same horizontal line) $\ell_2$ crosses $h(\ell)$ and the part of $v(\ell)$ to the left of $\ell_1'$ is equal to the part of $v(\ell)$ to the left of $\ell_2$.
Therefore, for any $\ell\in\famL$, Conditions \ref{cond:left1} and~\ref{cond:left2} for the pair $(\ell,\ell_1')$ follow from Conditions \ref{cond:left1} and~\ref{cond:left2} for the pair $(\ell,\ell_2)$.
It follows from the above that $\famL'$ satisfies Conditions \ref{cond:left1} and~\ref{cond:left2}.
In particular, $\famL'\cup\famS$ is triangle-free (in every triangle, the corner of one L-shape lies to the left of the other two members, contradicting Condition~\ref{cond:left1}).
We set $\famL:=\famL'$ and repeat the whole process for a next pair $\ell_1,\ell_2\in\famL$ occurring in configuration (e) or (f), until no such pair remains.
This must terminate, as the set of horizontal lines containing the top endpoints of the L-shapes in $\famL$ does not change.
The final family $\famL$ satisfies Condition~\ref{cond:vertical}.
\end{reduction}

\begin{condition}
\label{cond:horizontal}
No two L-shapes in $\famL$ occur in configuration (g) from Figure \ref{fig:conf}.
\end{condition}

\begin{reduction}
Suppose $\ell_1,\ell_2\in\famL$ occur in configuration (g), where $\ell_1$ has lower horizontal segment.
Pull the right endpoint of $\ell_1$ further to the right onto the rightmost support of $\ell_2$ to obtain an L-shape $\ell_1'$ such that $h(\ell_1)\subset h(\ell_1')$ and $v(\ell_1)=v(\ell_1')$.
Let $\famL'=\famL\setminus\{\ell_1\}\cup\{\ell_1'\}$.
Clearly, $\famL'$ satisfies Conditions \ref{cond:handle}--\ref{cond:below} (Condition~\ref{cond:below} for the right endpoint of $\ell_1'$ follows from Conditions \ref{cond:hook} and~\ref{cond:below} for the right endpoint of $\ell_2$).
Let $s$ be the leftmost support of $\ell_2$.
By Condition~\ref{cond:below}, $h(\ell_1)$ and $s$ intersect.
Suppose that $h(\ell_1')\setminus h(\ell_1)$ intersects some $\ell'\in\famL\cup\famS$.
Every point on $v(\ell_1)$ to the left of $\ell'$ lies to the left of $s$ or to the left of both $\ell_2$ and $\ell'$ (in the latter case, $\ell_2$ and $\ell'$ intersect).
Therefore, Conditions \ref{cond:left1} and~\ref{cond:left2} for the pair $(\ell_1',\ell')$ follow from Conditions \ref{cond:left1} and~\ref{cond:left2} for the pairs $(\ell_1,s)$ and $(\ell_2,\ell')$.
It follows that $\famL'$ satisfies Conditions \ref{cond:left1} and~\ref{cond:left2}; in particular, $\famL'\cup\famS$ is triangle-free.
Now, suppose that $\ell_1'$ and some $\ell'\in\famL$ are in configuration (e) or (f) from Figure~\ref{fig:conf}.
Since $\ell_1$ and $\ell'$ are not (by Condition~\ref{cond:vertical}), $\ell_1$ and $v(\ell_2)$ lie to the left of the vertical line containing $v(\ell')$.
Since $\ell_1$ and $\ell_2$ are in configuration (g), while $\ell_2$ and $\ell'$ are not in configuration (a), (b), (e), or (f) (by Conditions \ref{cond:hook}, \ref{cond:left2}, and~\ref{cond:vertical}), $\ell_2$ lies entirely below the horizontal line containing $h(\ell')$.
Therefore, $v(\ell_2)$ lies to the left of the leftmost support of $\ell_2'$ while $h(\ell_2)$ intersects that support (by Condition~\ref{cond:below}), which contradicts Condition~\ref{cond:left2}.
This shows that $\famL'$ satisfies Condition~\ref{cond:vertical}.
We set $\famL:=\famL'$ and repeat the whole process for a next pair $\ell_1,\ell_2\in\famL$ occurring in configuration (g), until no such pair remains.
This must terminate, as the set of vertical lines containing the right endpoints of the L-shapes in $\famL$ does not change.
The final family $\famL$ satisfies Condition~\ref{cond:horizontal}.
\end{reduction}

Now, of all configurations illustrated in Figure~\ref{fig:conf}, only (d) and (h) can occur in $\famL$, and the desired bound $\chi(\famL)=O(\log\log{|\famL|})$ follows from Lemma~\ref{lem:frames}.
This completes the proof of Lemma~\ref{lem:key}.
\end{proof}

\section{Proof of Theorem \ref{thm:main}}

Let $\famF_0$ be a triangle-free family of $n$ L-shapes.
Without loss of generality (performing an appropriate perturbation of the L-shapes in $\famF_0$ if necessary), we can assume that
\begin{itemize}
\item the horizontal segments of the L-shapes in $\famF_0$ have distinct $y$-coordinates,
\item the vertical segments of the L-shapes in $\famF_0$ have distinct $x$-coordinates.
\end{itemize}
We construct subfamilies $\famF_1$, $\famF_2$, and $\famF_3$ so that
\begin{itemize}
\item $\famF_0\supset\famF_1\supset\famF_2\supset\famF_3$,
\item $\chi(\famF_i)=O(\chi(\famF_{i+1})+\log\log n)$ for each $i\in\{0,1,2\}$,
\item $\chi(\famF_3)$ is bounded by a constant.
\end{itemize}
These conditions imply that $\chi(\famF_0)=O(\log\log n)$.

Let $i\in\{0,1,2\}$.
We construct $\famF_{i+1}$ from $\famF_i$ as follows.
By Lemma~\ref{lem:radius2}, there is an L-shape $\ell_i\in\famF_i$ such that $\chi(\famF_i)\leq\alpha\chi(\famF_i(\ell_i))$ for some absolute constant $\alpha$.
Let the L-shapes in $\famF_i\setminus\{\ell_i\}$ that cross $h(\ell_i)$ be called \emph{$h(\ell_i)$-supports} and those that cross $v(\ell_i)$ be called \emph{$v(\ell_i)$-supports}.
The $h(\ell_i)$-supports and $v(\ell_i)$-supports are pairwise disjoint, as $\famF_i$ is triangle-free.
Every L-shape in $\famF_i(\ell_i)$ crosses at least one $h(\ell_i)$-support or $v(\ell_i)$-support.
Let $h_i$ be the horizontal line containing $h(\ell_i)$ and $v_i$ be the vertical line containing $v(\ell_i)$.
Define subfamilies $\famF_i^1,\ldots,\famF_i^6$ of $\famF_i(\ell_i)$ as follows:
\begin{align*}
\famF_i^1&=\{\ell\in\famF_i(\ell_i)\colon\text{some $h(\ell_i)$-support crosses $h(\ell)$, and $h(\ell)$ lies above $h_i$}\},\\
\famF_i^2&=\{\ell\in\famF_i(\ell_i)\colon\text{some $h(\ell_i)$-support crosses $h(\ell)$, and $h(\ell)$ lies below $h_i$}\},\\
\famF_i^3&=\{\ell\in\famF_i(\ell_i)\colon\text{some $h(\ell_i)$-support crosses $v(\ell)$, but no $h(\ell_i)$-support crosses $h(\ell)$}\},\\
\famF_i^4&=\{\ell\in\famF_i(\ell_i)\colon\text{some $v(\ell_i)$-support crosses $v(\ell)$, and $v(\ell)$ lies to the right of $v_i$}\},\\
\famF_i^5&=\{\ell\in\famF_i(\ell_i)\colon\text{some $v(\ell_i)$-support crosses $v(\ell)$, and $v(\ell)$ lies to the left of $v_i$}\},\\
\famF_i^6&=\{\ell\in\famF_i(\ell_i)\colon\text{some $v(\ell_i)$-support crosses $h(\ell)$, but no $v(\ell_i)$-support crosses $v(\ell)$}\}.
\end{align*}
It follows that $\famF_i(\ell_i)=\famF_i^1\cup\cdots\cup\famF_i^6$ and thus $\chi(\famF_i(\ell_i))\leq\chi(\famF_i^1)+\cdots+\chi(\famF_i^6)$.

First, we prove that $\chi(\famF_i^1)=O(\log\log n)$.
Let $\famL=\famF_i^1$.
Every L-shape in $\famL$ lies above $h_i$.
Let $\famS$ be the $h(\ell_i)$-supports with the parts lying below $h_i$ removed (they are vertical segments starting at $h_i$).
It follows that $h_i$, $\famL$, and $\famS$ satisfy the assumptions of Lemma~\ref{lem:key}, which then implies that $\chi(\famL)=O(\log\log{|\famL|})=O(\log\log n)$.
By symmetry, we also have $\chi(\famF_i^4)=O(\log\log n)$.

Next, we prove that $\chi(\famF_i^6)=O(\log\log n)$.
Let $\famL=\famF_i^6$.
Every L-shape in $\famL$ lies above $h_i$.
Let $\famS$ be the $v(\ell_i)$-supports with the parts lying to the right of $v_i$ removed (they are L-shapes with right endpoints at $v_i$).
Let $s\in\famS$.
By the definition of $\famF_i^6$ and the fact that the $v(\ell_i)$-supports are pairwise disjoint, no L-shape in $\famL$ or $\famS$ crosses $h(s)$.
Cut the plane along $h(s)$ and the rays emanating from the left endpoint of $h(s)$ upwards and from the right endpoint of $h(s)$ downwards, and shift the two parts horizontally, like in Figure~\ref{fig:shift}, keeping $s$ and $\ell_i$ in the right part and adding appropriate horizontal connections to the L-shapes in $\famL$ and $\famS$ that have been cut through.
This ensures that no part of an L-shape in $\famL$ lies below $h(s)$.
Perform such a horizontal shift for every $s\in\famS$.
Then, replace each $s\in\famS$ by the vertical extension of $v(s)$ down to $h_i$ (it intersects no more L-shapes in $\famL$).
The horizontal line $h_i$ and the families $\famL$ and $\famS$ thus obtained satisfy the assumptions of Lemma~\ref{lem:key}, which then implies that $\chi(\famL)=O(\log\log{|\famL|})=O(\log\log n)$.
By~symmetry, we also have $\chi(\famF_i^3)=O(\log\log n)$.

Let $k_i\in\{2,5\}$ be such that $\chi(\famF_i^{\smash[b]{k_i}})=\max\{\chi(\famF_i^2),\chi(\famF_i^5)\}$, and let $\famF_{i+1}=\famF_i^{\smash[b]{k_i}}$.
It follows that $\chi(\famF_i)\leq\alpha\chi(\famF_i(\ell_i))=O(\chi(\famF_{i+1})+\log\log n)$.
It remains to prove that $\chi(\famF_3)$ is bounded.

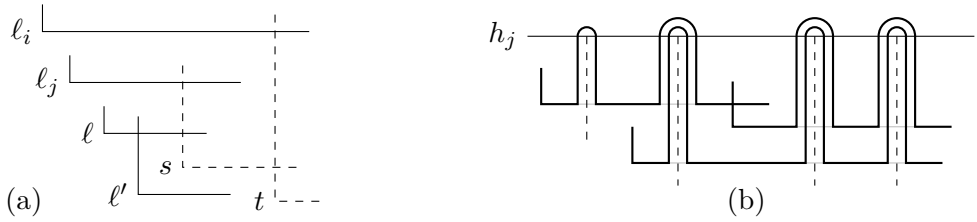
\begin{figure}[t]
\begin{tikzpicture}[scale=0.45]
  \draw (-0.8,2.3)--(-0.8,1.5) node[left] {$\ell_i$}--(7,1.5);
  \draw[dashed] (6,2)--(6,-3.5) node[left] {$t$}--(7.5,-3.5);
  \draw (0,0.8)--(0,0) node[left] {$\ell_j$}--(5,0);
  \draw[dashed] (3.3,0.5)--(3.3,-2.5) node[left] {$s$}--(6.8,-2.5);
  \draw (1,-0.7)--(1,-1.5) node[left] {$\ell$}--(4,-1.5);
  \draw (2,-1)--(2,-3.3) node[left] {$\ell'$}--(4.7,-3.3);
  \node[left] at (-0.5,-3.5) {(a)};
\end{tikzpicture}\hskip 2cm
\begin{tikzpicture}[scale=0.6,yscale=-1]
  \draw (0.2,0)--(10,0);
  \draw[black!25] (1.3,1.5)--(1.7,1.5) (3.1,1.5)--(3.9,1.5);
  \draw[black!25] (6.1,2)--(6.9,2) (7.9,2)--(8.7,2);
  \draw[black!25] (3.3,2.8)--(3.7,2.8) (6.3,2.8)--(6.7,2.8) (8.1,2.8)--(8.5,2.8);
  \draw[thick] (0.5,0.7)--(0.5,1.5)--(1.3,1.5)--(1.3,0) arc (-180:0:2mm)--(1.7,1.5)--(3.1,1.5)--(3.1,0) arc (-180:0:4mm)--(3.9,1.5)--(5.5,1.5);
  \draw[thick] (4.7,1)--(4.7,2)--(6.1,2)--(6.1,0) arc (-180:0:4mm)--(6.9,2)--(7.9,2)--(7.9,0) arc (-180:0:4mm)--(8.7,2)--(9.5,2);
  \draw[thick] (2.5,2)--(2.5,2.8)--(3.3,2.8)--(3.3,0) arc (-180:0:2mm)--(3.7,2.8)--(6.3,2.8)--(6.3,0) arc (-180:0:2mm)--(6.7,2.8)--(8.1,2.8)--(8.1,0) arc (-180:0:2mm)--(8.5,2.8)--(9.3,2.8);
  \draw[dashed] (1.5,0)--(1.5,2.3);
  \draw[dashed] (3.5,0)--(3.5,3.3);
  \draw[dashed] (6.5,0)--(6.5,3.3);
  \draw[dashed] (8.3,0)--(8.3,3.3);
  \node[left] at (0.2,0) {$h_j$};
  \node[below] at (5,3.1) {(b)};
\end{tikzpicture}
\vspace*{-2ex}
\caption{Final part of the proof of Theorem~\ref{thm:main}}
\label{fig:final}
\end{figure}

There are indices $i,j\in\{0,1,2\}$ such that $i<j$ and $k_i=k_j\in\{2,5\}$.
Suppose $k_i=k_j=2$.
The \emph{hook} of an L-shape $\ell\in\famF_3$ is the part of $\ell$ to the right of the rightmost intersection point of $h(\ell)$ with an $h(\ell_i)$-support (from $\famF_j$).
This is well defined---since $\ell\in\famF_j^2$, at least one $h(\ell_j)$-support crosses $h(\ell)$.
We claim that every crossing between two L-shapes in $\famF_3$ is a crossing between the vertical part of one and the hook of the other.
Suppose to the contrary that there are two intersecting L-shapes $\ell,\ell'\in\famF_3$ such that $v(\ell')$ intersects $h(\ell)$ but not the hook of $\ell$.
Let $s$ be the $h(\ell_j)$-support (from $\famF_j$) with rightmost intersection point with $h(\ell)$, so that $v(\ell')$ intersects $h(\ell)$ to the left of $s$.
Since $s\in\famF_i^2$, there is an $h(\ell_i)$-support $t\in\famF_i$ that crosses $h(s)$; see Figure \ref{fig:final}~(a).
The segment $h(\ell_j)$ lies entirely to the left of $t$, otherwise $\ell_j$, $s$, and $t$ would form a triangle.
Now, any $h(\ell_j)$-support (from $\famF_j$) that crosses $h(\ell')$ must intersect $\ell$ (forming a triangle with $\ell$ and $\ell'$) or $s$ (forming a triangle with $\ell_j$ and $s$).
This contradiction shows that every crossing between two L-shapes in $\famF_3$ is a crossing between the vertical part of one and the hook of the other.
Consequently, the family $\famF_3$ can be transformed into an ``upside-down'' $LR$-family of $h_j$-even-curves with the same intersection graph, as illustrated in Figure \ref{fig:final}~(b).
Therefore, by Lemma~\ref{lem:even}, $\chi(\famF_3)$ is bounded.
This completes the proof of Theorem~\ref{thm:main} for the case $k_i=k_j=2$.
The proof for the case $k_i=k_j=5$ is analogous, by symmetry.

\end{document}